\documentclass[12pt,reqno]{amsart}
\usepackage{amsfonts}
\usepackage{amsthm}
\usepackage{amsmath}
\usepackage{amscd}
\numberwithin{equation}{section}
\usepackage{hyperref}
\hypersetup{nesting=true,debug=true,naturalnames=true}
\usepackage{graphicx,amssymb,upref}
\usepackage{enumerate}
\usepackage[all]{xy}
\usepackage{color}
\usepackage{mathrsfs}
\usepackage{caption}
\usepackage{paralist}
\usepackage{color}
\usepackage{float}
\usepackage{tikz,pgf}
\usetikzlibrary{calc}
\usetikzlibrary{intersections,decorations.markings}
\usepackage{tikz-cd}
\usepackage{lineno}

\theoremstyle{definition}
\newtheorem{theorem}{\bf Theorem}[section]

\newtheorem{lemma}[theorem]{\bf Lemma}

\theoremstyle{definition}

\newtheorem{definition}[theorem]{\bf Definition}

\newtheorem{proposition}[theorem]{\bf Proposition}

{

}
\newcommand{\mm}[1]{\mathrm{#1}}
\newcommand{\mb}[1]{\mathbb{#1}}
\newcommand{\mc}[1]{\mathcal{#1}}
\makeatletter
\@namedef{subjclassname@2020}{\textup{2020} Mathematics Subject Classification}
\makeatother

\begin{document}

\title[Classification for smooth manifolds looking like $\mathbb{CP}^3\times S^7$]{Classification for smooth manifolds looking like $\mathbb{CP}^3\times S^7$}

\author[Wen Shen]{Wen Shen}
\address{College of Mathematics and Physics, Wenzhou University, Wenzhou, P.R.China}
\email{shenwen121212@163.com}

\begin{abstract}
In this paper, we classify simply connected closed smooth $13$-dimensional manifolds whose cohomology ring is isomorphic to that of $\mb{CP}^3\times S^7$, up to diffeomorphism, homeomorphism, and homotopy equivalence. Furthermore, if such a manifold satisfies certain conditions, either itself or its connected sum with an exotic 
$13$-sphere $\Sigma^{13}$ admits a Riemannian metric of non-negative sectional curvature. As an additional application of our classification, we classify the diagonal $S^1$-actions on $S^7\times S^7$.    
\end{abstract}


\maketitle

\section{Introduction}\label{intro}

Classification of manifolds with prescribed cohomology rings (up to diffeomorphism, homeomorphism, or homotopy equivalence) constitutes a central problem in geometric topology. The Poincar\'{e} conjecture provides a seminal example: for 
$n\ge 3$, any $n$-dimensional manifold sharing the homology groups of $S^n$ is homeomorphic to the 
$n$-sphere. In the general case, Wall \cite{Wall1962,Wall1967} investigated $(s-1)$-connected $2s$-manifolds and $(s-1)$-connected $(2s+1)$-manifolds.
 Specific dimensional cases yield richer classification results: Barden \cite{Barden} achieved complete classification of simply connected 5-manifolds, while Kreck and Su \cite{KreckSu} considered certain nonsimply connected 5-manifolds. One of their results is the classification for closed oriented $5$-manifolds $M$ with $\pi_1(M)$ a free group and $H_2(M;\mb{Z}) = 0$.
In dimension 6, Wall \cite{Wall1966} classified closed  simply connected spin manifolds with torsion-free homology, later extended by Jupp \cite{Jupp} to non-spin cases. Kreck and Stolz \cite{KS1, KS} gave a classification for the $7$-manifolds modeled on Aloff-Wallach spaces \cite{Aloff}, notable as smooth manifolds admitting positive sectional curvature. The homotopy classification of Aloff-Wallach spaces was given by Kruggel \cite{Kruggel}. Further classifications for 7-manifolds with specialized cohomology rings appear in \cite{CroNor,K2018}.

Throughout this paper, let $H^i(-)$ and $H_i(-)$ denote the integral cohomology and homology groups $H^i(-;\mathbb{Z})$ and $H_i(-;\mathbb{Z})$ respectively unless otherwise specified.

We focus on simply connected, closed, smooth $13$-manifolds ${M}$ satisfying $H^\ast({M})\cong H^\ast(\mb{CP}^3\times S^7)$ where 
$\mb{CP}^ 
3$ denotes the 
$3$-dimensional complex projective space. A canonical family of such manifolds arises from the circle actions on $S^7\times S^7$ \cite{WZ}
defined by $\theta \cdot  (x,y) = (e^{il\theta}x,e^{-ik\theta}y)$ with $\mm{gcd}(k,l)=1$. Certain considered manifolds exhibit rich geometric properties:
From \cite{DDGR}, there are infinitely many circle quotients of $S^7 \times S^7$ admitting metrics with $\mm{Ric}_2 > 0$. 
Furthermore, Kerin \cite{Kerin} gave an example with almost positive sectional curvature.

 Before stating the main theorem, we first introduce $13$-dimensional homotopy spheres, denoted as $\Sigma^{13}$. A homotopy sphere $\Sigma^{13}$ is a smooth manifold that is homotopy equivalent to the standard sphere $S^{13}$. By Poincar\'{e} conjecture, $\Sigma^{13}$ is homeomorphic to $S^{13}$. We define $\Sigma^{13}$ to be an exotic sphere if $\Sigma^{13}$ is not diffeomorphic to $S^{13}$. 
 The diffeomorphism classes of homotopy spheres $\Sigma^{13}$ form a group $\Theta_{13}\cong \mb{Z}_3$ \cite{KervaMilnor} under the operation of connected sum. The standard sphere $S^{13}$ corresponds to the $0$ element of this group, arbitrary exotic sphere $\Sigma^{13}$ represents a generator of $\Theta_{13}$. 

Recall the ring structure $H^\ast( M)\cong H^\ast(\mb{CP}^3\times S^7)$. Note that $x^2=-x\cup -x$ where $x$ is a generator of $H^2( M)$.
We consistently choose the cup product $x^2$ as the generator of $H^4( M)$. Consequently, we can represent the first Pontrjagin class  $p_1(M)$ of $M$ using just an integer.

Now that we have laid the necessary groundwork, we are ready to state the main theorem of this paper.
 
 \begin{theorem}\label{1.1}
{\it	Let $ M$, $ M^\prime$ be simply connected, closed, smooth $13$-manifolds with the same cohomology ring as $H^\ast(\mb{CP}^3\times S^7)$. 
\begin{enumerate}
	\item $M$ is homeomorphic to $M^\prime$ if and only if $p_1( M)=p_1(M^\prime)$. 
		\item  If $p_1( M)=p_1( M^\prime)\not\equiv 0\pmod 2$, there exists a homotopy sphere $\Sigma^{13}$ so that $ M$ is diffeomorphic to $ M^\prime \# \Sigma^{13}$.		
		\item If $p_1( M)=p_1( M^\prime)$ is coprime to $6$, $ M$ is diffeomorphic to $ M^\prime$.
		
		\item $ M$ is homotopy equivalent to $ M^\prime$ if and only if $p_1( M)\equiv p_1( M^\prime)\pmod {24}$.
		\item If $p_1( M)=p_1( M^\prime)\equiv 4\pmod{24}$, $ M$ is diffeomorphic to $ M^\prime$.
\end{enumerate}
}
\end{theorem}

We primarily employ surgery theory \cite{K} to establish the proofs of (1), (2) and (3) in Theorem \ref{1.1}. Item (4) follows from the homotopy classification in \cite{Shen}.
If a manifold $M$ as in Theorem \ref{1.1} satisfies $p_1(M)=4$, then $M$ is homeomorphic to $\mb{CP}^3\times S^7$ by Theorem \ref{1.1} (1). By \cite[Theorem 3.33]{BasuKaSa}, it is diffeomorphic to $\mb{CP}^3\times S^7$. Here, we combine the method in \cite{BasuKaSa} with the spectral sequence to prove Theorem \ref{1.1} (5).

In Riemannian geometry, established results demonstrate how curvature governs the topological structure of manifolds. Classical theorems such as the \textit{Bonnet-Myers Theorem} and the \textit{Synge Theorem} exemplify this geometric control principle.
Conversely, a fundamental question persists regarding the inverse direction: Which smooth manifolds admit metrics with $\mm{sec}\ge 0$ or $\mm{sec}>0$? This paper contributes to this inquiry by developing topological criteria that constrain potential metric realizations. Specifically, we establish:

\begin{theorem}\label{1.3}
	{\it Let $M$ be a simply connected, closed, smooth manifold with the same cohomology ring as $H^\ast(\mb{CP}^3\times S^7)$. Assume $p_1({M})\ge 4$.
	\begin{enumerate}
		\item If $w_2(M)\ne 0$, then there exists a homotopy $13$-sphere $\Sigma^{13}$ such that ${M}\#\Sigma^{13}$ admits a Riemannian metric of non-negative sectional curvature.
		\item If $\mm{gcd}(p_1({M}), 6)=1$ or $p_1({M})\equiv 4\pmod{24}$, then ${M}$ itself admits a Riemannian metric of non-negative sectional curvature.
	\end{enumerate}
	}
\end{theorem} 

Finally, we consider certain $S^1$ actions on $S^7\times S^7$, whose quotient spaces are simply connected closed smooth manifolds with the same cohomology ring as $H^\ast(\mb{CP}^3\times S^7)$. 

Let $\bold a=(a_1,a_2,a_3,a_4)$, $\bold b=(b_1,b_2,b_3,b_4)\in \mb{Z}^4$.   
Given an $S^1_{\bold{a,b}}$-action on $S^7\times S^7$ by 
$\theta \cdot ((x_1,x_2,x_3,x_4),(y_1,y_2,y_3,y_4)):=$
$$((e^{\bold{i}a_1\theta}x_1,e^{\bold{i}a_2\theta}x_2,e^{\bold{i}a_3\theta}x_3,e^{\bold{i}a_4\theta}x_4),(e^{\bold{i}b_1\theta}y_1,e^{\bold{i}b_2\theta}y_2,e^{\bold{i}b_3\theta}y_3,e^{\bold{i}b_4\theta}y_4))$$
where $\theta\in S^1\cong \mb{R}/\{2k\pi\}$ ($k\in \mb{Z}$), $\Sigma_i ||x_i||^2=\Sigma_i ||y_i||^2=1$ ($x_i,y_i\in \mb{C}$).

\begin{proposition}\label{freeS1action}
	 {\it The $S^1_{\bold{a,b}}$-action on $S^7\times S^7$ is a free action if and only if $a_i,b_j\ne 0$ and $\mm{gcd}(a_i,b_j)=1$ for any $1\le i,j\le 4$.}
\end{proposition} 

\begin{definition}
	{\it Let $G$ be a group, $X$ be a space, $\rho_1,\rho_2:G\times X\to X$ be $G$-actions on $X$. If there exists a self-homeomorphism $h:X\to X$ such that $h(\rho_1(g,x))=\rho_2(g,h(x))$ for any $g\in G$, $x\in X$, we call that $\rho_1$ and $\rho_2$ are equivalent. If the $X$ is a smooth manifold and the $h$ is a self-diffeomorphism, we call that $\rho_1$ and $\rho_2$ are smoothly equivalent.}   
\end{definition}

\begin{theorem}\label{classifyS1action}
	{\it  Assume the $S^1_{\bold{a,b}}$, $S^1_{\bold{\bar a,\bar b}}$-actions on $S^7\times S^7$ are free. 
	 \begin{enumerate}
	 	\item If $\Sigma_i (a_i^2+b_i^2)=\Sigma_i (\bar a_i^2+\bar b_i^2)$, the $S^1_{\bold{a,b}}$-action is equivalent to the $S^1_{\bold{\bar a,\bar b}}$ or $S^1_{\bold{-\bar a,-\bar b}}$-action.
	 	\item If $\Sigma_i (a_i^2+b_i^2)=\Sigma_i (\bar a_i^2+\bar b_i^2)$ is coprime to $6$ or congruent to $4$ modulo $24$, then the $S^1_{\bold{a,b}}$-action is smoothly equivalent to the $S^1_{\bold{\bar a,\bar b}}$ or $S^1_{\bold{-\bar a,-\bar b}}$-action.
	 \end{enumerate}}	 
\end{theorem}

 We first make some basic constructions in Section \ref{Preli}. In Section \ref{normal5}, we develop a suitable bordism theory within the smooth category for the manifolds under consideration. Subsequently, in Section \ref{FiltrationforM}, we analysis a certain bordism group.
 Thereafter, in Section \ref{PLcateSec}, we carry out the computation of a certain bordism group in PL category. The proofs of Theorem \ref{1.1} are furnished in Section \ref{proofmain}. Finally, in Section \ref{constructforXab}, we explore the diagonal $S^1$-actions, thus prove Proposition \ref{freeS1action} and Theorem \ref{classifyS1action}.

\section{Preliminary}\label{Preli}

Let $M$ be a simply connected, closed, smooth $13$-manifold with the same cohomology ring as $ H^\ast(\mb{CP}^3\times S^7)$.

By \cite[Theorem 4.57]{Hatcher}, cohomology classes in $H^2({M})$ bijectively correspond to homotopy classes $[{M},\mm{K}(\mb{Z},2)]$ where $\mm{K}(\mb{Z},2)$ denotes the Eilenberg-MacLane space. $\mm{K}(\mb{Z},2)$ is homotopy equivalent to $\mb{CP}^\infty$ that is the colimit over the sequence $\mb{CP}^1\hookrightarrow \cdots \hookrightarrow \mb{CP}^n\hookrightarrow \mb{CP}^{n+1}\hookrightarrow \cdots$ where $\mb{CP}^n\hookrightarrow \mb{CP}^{n+1}$ is the natural inclusion. Hence there exists a bijection $H^2({M})\to [M, \mb{CP}^\infty]$.

Consider the map $f:{M}\to \mb{CP}^\infty$ corresponding to a generator of $H^2({M})$. Pulling back the universal principal $S^1$-bundle $\gamma$ over $\mb{CP}^\infty$ yields the following $S^1$-bundle over ${M}$
\begin{equation}
	f^\ast\gamma 
	 :S^1\hookrightarrow E\to {M} \label{S1bunM}
\end{equation}
  
\begin{lemma}
	\label{Globalcha}
	The manifold $E$ is homeomorphic to $S^7\times S^7$.
\end{lemma}
\begin{proof}
 The bundle \eqref{S1bunM} implies that $E$ is a smooth simply connected $14$-manifold.
By the Gysin sequence \cite[pp.438]{Hatcher}, we have 
	  $$H^\ast(E)\cong H^\ast(S^7\times S^7)$$
	  	  
	  Let $E \backslash \mathring D^{14}$  denote the manifold obtained by removing the interior of a $14$-dimensional closed disk from the manifold $E$. Obviously, its boundary $\partial( E \backslash \mathring D^{14})$ is the standard sphere $S^{13}$. Applying the homotopy and homology exact sequences for the pair $(E,E \backslash \mathring D^{14})$, we have 
	  \begin{enumerate}
	  	\item $E \backslash \mathring D^{14}$ is a smooth $6$-connected $14$-manifold. 
	  	\item $H_7(E \backslash \mathring D^{14})$ is isomorphic to $\mb{Z}\oplus \mb{Z}$.
	  \end{enumerate}
	  	 
	   By Wall's classification  \cite[Lemma 5, Case 7]{Wall1962},  
	   $E \backslash \mathring D^{14}$ is diffeomorphic to $(S^7\times S^7)\backslash\mathring D^{14}$. Moreover, the boundary diffeomorphism 	 $$\partial(E\backslash \mathring D^{14})=S^{13}\to \partial((S^7\times S^7)\backslash \mathring D^{14})=S^{13}$$
	 extends to a homeomorphism $D^{14}\to D^{14}$ by radial extension. 
	Consequently, $E$ is homeomorphic to $S^7\times S^7$.
	\end{proof}	

  Now we construct a fibre bundle:
\begin{equation}
	E\hookrightarrow S^\infty\times_{S^1}E\to \mb{CP}^\infty \label{fibrebun-S7S7overCPin}
\end{equation}
where the quotient space $S^\infty\times_{S^1}E$ is induced by the $S^1$-action on $S^\infty\times E$: the right $S^1$-action on $S^\infty$ is the Hopf action, the left $S^1$-action on $E$ is induced by the principle $S^1$-bundle \eqref{S1bunM}. A straightforward verification shows that $S^\infty\times_{S^1}E$ is homotopy equivalent to $M$.

By the homeomorphism $E\cong S^7\times S^7$, we have the following bundle isomorphism
\[
\xymatrix@C=1.8cm{
E\ar[d]^-{}\ar[r]^-{\cong} &S^7\times S^7\ar[d]^-{}\\
S^\infty\times_{S^1}E\ar[d] \ar[r]^-{\cong}& S^\infty\times_{S^1} (S^7\times S^7)\ar[d]\\
\mb{CP}^\infty\ar@{=}[r]& \mb{CP}^\infty}
\] 
where the $S^1$-action on $S^7\times S^7$ is induced by the $S^1$-action on $E$. 

\begin{definition}\label{aigenera}
	Let $a_i$ ($i=1,2$) be a basis of $H^7(S^7\times S^7)$ satisfying 
	$a_i\in \mm{p}_i^\ast(H^7(S^7))$ where $\mm{p}_i$ is the projection $S^7\times S^7\to S^7$ for the $i$-th factor.
\end{definition}

For the fibre bundle 
$S^7\times S^7\hookrightarrow S^\infty\times_{S^1}(S^7\times S^7)\to \mb{CP}^\infty$, we have the Serre spectral sequence
\begin{equation}
	E_2^{p,q}=H^p(\mb{CP}^\infty;H^q(S^7\times S^7))\Longrightarrow H^{p+q}(S^\infty\times_{S^1}(S^7\times S^7)) \label{SSSforS7timesS7}
\end{equation}

Recall that $H^\ast(\mb{CP}^\infty)\cong \mb{Z}[x]$.
\begin{lemma}\label{diffforS7timesS7}
In the spectral sequence \eqref{SSSforS7timesS7},
\begin{enumerate}
	\item $E_8^{0,7}=E_2^{0,7}$, $E_8^{8,0}=E_2^{8,0}$, and $E_\infty^{8,0}=E_9^{8,0}$.
	\item Let $d_8(a_i)=s_ix^8$ where $s_i\in \mb{Z}$, $i=1,2$. If $s_i=0$, then $s_{3-i}=\pm 1$. If $s_1\ne 0$ and $s_2\ne 0$,  then $s_1$ is coprime to $s_2$.
\end{enumerate}
\end{lemma}
\begin{proof}
	The item (1) is obvious. We only prove item (2).
	
	Since $S^\infty\times_{S^1}(S^7\times S^7)$ is homeomorphic to $S^\infty\times_{S^1}E$, $S^\infty\times_{S^1}(S^7\times S^7)$ is homotopy equivalent to $M$. Thus $H^8(S^\infty\times_{S^1}(S^7\times S^7))=0$. This implies $E_\infty^{8,0}=0$. So there exist $r_i\in \mb{Z}$ ($i=1,2$) such that 
	$$d_8(r_1a_1+r_2a_2)=x^4$$
	In other words, $r_1s_1+r_2s_2=1$. This finishes the proof.  
\end{proof}

\begin{proposition}\label{Xconstruct}
	Let $M$ be a simply connected, closed, smooth $13$-manifold with the same cohomology ring as $ H^\ast(\mb{CP}^3\times S^7)$.  For any prime $p$, there exists a space $X$ such that
	\begin{enumerate}
		\item $H^\ast(X)\cong \mb{Z}[x]/kx^4$ where $x\in H^2(X)$, $0\ne k$ is coprime to $p$.
		\item There exists a map $g:M\to X$ inducing isomorphisms on $H^i$ for $0\le i\le 6$. 
		\end{enumerate}
\end{proposition}
\begin{proof}
For $i=1$ or $2$,
the projection $\mm{p}_i:S^7\times S^7\to S^7$ induces the following bundle morphism
	\[
\xymatrix@C=1.8cm{
S^7\times S^7\ar[d]^-{}\ar[r]^-{\mm{p}_i} & S^7\ar[d]^-{}\\
S^\infty\times_{S^1}(S^7\times S^7)\ar[d] \ar[r]^-{\bar{\mm{p}}_i}& X=S^\infty\times_{S^1} S^7\ar[d]\\
\mb{CP}^\infty\ar@{=}[r]& \mb{CP}^\infty}
\]
where the quotient space $S^\infty\times_{S^1}S^7$ is induced by the $S^1$-action on $S^\infty\times S^7$: the right $S^1$-action on $S^\infty$ is the Hopf action, the left $S^1$-action on $S^7$ is induced by the projection $\mm{p}_i$.

The above bundle morphism induces a morphism between the Serre spectral sequences
\[
\xymatrix@C=.2cm{
  \tilde{E}_2^{p,q}={H}^p(\mb{CP}^\infty;H^q(S^7))\ar@{=>}[rr]^-{\tilde d_r}\ar[d]&  & H^{p+q}(X)\ar[d]^-{} \\
{E}_2^{p,q}={H}^p(\mb{CP}^\infty;H^q(S^7\times S^7))\ar@{=>}[rr]^-{d_r}& & H^{p+q} ( S^\infty\times_{S^1}(S^7\times S^7))
}
\] 
Let $b_i\in H^7(S^7)$ satisfy $\mm{p}_i^\ast(b_i)=a_i\in H^7(S^7\times S^7)$. Suppose $d_8(a_i)=s_ix^8$, then $\tilde d_8(b_i)=s_ix^8$ by the naturality of the Serre spectral sequence. 
Thus, by a straightforward computation, we have 
 $H^\ast(X)\cong \mb{Z}[x]/s_ix^4$ where $\mm{deg}(x)=2$. 

By Lemma \ref{diffforS7timesS7} (2), we consider the following three cases: 
\begin{itemize}
	\item If $s_1=0$, then $s_2=1$. We take $i=2$ in the above process. Thus $H^\ast(X)\cong \mb{Z}[x]/x^4$.
	\item If $s_2=0$, then $s_1=1$. We take $i=1$ in the above process. Thus $H^\ast(X)\cong \mb{Z}[x]/x^4$.
	\item If $s_1\ne 0$ and $s_2\ne 0$, then $s_1$ is coprime to $s_2$. Hence, for any prime number $p$, either $s_1$ or $s_2$ is coprime to $p$. Suppose $s_1$ is coprime to $p$, then we take $i=1$ in the above process. Thus $H^\ast(X)\cong \mb{Z}[x]/s_1x^4$.
\end{itemize}
This finishes the proof of (1).  

Let $g$ be the composition $M\simeq S^\infty\times_{S^1}(S^7\times S^7)\stackrel{\bar{\mm{p}}_i}{\to }X$. It is easy to check that $g$ induces isomorphisms on $H^i$ for $0\le i\le 6$.
\end{proof}

\section{The normal $B$-structure}\label{normal5}
Let $\pi: B \to \mathrm{BO}$ be a fibration over the classifying space $\mathrm{BO}$ of stable vector bundles. 
\begin{definition}\label{DefnorB}
	(1) A stable vector bundle admits a $B$-structure if its classifying map has a lift to $B$.
	
	(2) A normal $B$-structure of a smooth manifold $M$ is a lift of the classifying map of its stable normal bundle to $B$.
\end{definition}


Let ${M}$ be a simply connected, closed, smooth 13-dimensional manifold with cohomology ring $H^*({M}) \cong H^*(\mathbb{CP}^3 \times S^7)=\mb{Z}[x,y]/(x^4,y^2)$ where $\mm{
deg}(x)=2$, $\mm{deg}(y)=7$. 
The $x$ corresponds to a map $f:M\to \mb{CP}^\infty$ that induces isomorphisms on $H^i$ for $0\le i\le 6$.

 Suppose the first Pontryagin class $p_1({M}) = nx^2 \in H^4({M})$,  
     the second Stiefel-Whitney class $w_2({M}) \equiv nx \ (\mathrm{mod}\ 2) \in H^2({M};\mathbb{Z}/2)$
where $n\in \mb{Z}$. Next we construct a vector bundle $\xi$ over $\mathbb{CP}^\infty$ whose first Pontryagin class $p_1(\xi)$ depends algebraically on the parameter $n$.

Let $\mathcal{H}$ denote the Hopf bundle over $\mathbb{CP}^\infty$. If $n\ge 0$, let $\xi$ be the complementary bundle of the Whitney sum $n\mathcal{H}$; if $n<0$, let $\xi$ be the Whitney sum $-n\mathcal{H}$. These constructions can be formally expressed as $\xi = -n\mathcal{H}$. The characteristic classes of $\xi$ are given by:
\begin{itemize}
    \item First Chern class: $c_1(\xi) = -nx \in H^2(\mathbb{CP}^\infty)$
    \item Second Chern class: $c_2(\xi) = \frac{n(n+1)}{2}x^2 \in H^4(\mathbb{CP}^\infty)$
    \item First Pontryagin class: $p_1(\xi) = -nx^2 \in H^4(\mathbb{CP}^\infty)$
\end{itemize}
where $x$ generates $H^2(\mathbb{CP}^\infty)$.

Let $\mathscr{N}_{{M}}$ denote both the stable normal Gauss map ${M}\to \mm{BO}$ and the stable normal bundle of ${M}$. By computing the
first Pontryagin class and the second Stiefel-Whitney class of the virtual bundle $\mathscr{N}_{{M}}-f^\ast \xi$, we have $w_2(\mathscr{N}_{{M}}-f^\ast \xi)=0$ and $p_1(\mathscr{N}_{{M}}-f^\ast \xi)=0$. 
Next, we show that the bundle $\mathscr{N}_{{M}}-f^\ast \xi$ admits a $\mm{BO}\langle 8\rangle$-structure. 

Let $\mathrm{BO}\langle n \rangle$ be the $(n - 1)$-connected cover of $\mathrm{BO}$. That is to say, $\mathrm{BO}\langle n \rangle$ is $(n - 1)$-connected, and there exists a fibration $\mathrm{BO}\langle n \rangle \to \mathrm{BO}$ which induces isomorphisms on $\pi_k$ for $k\geq n$. It is well-known that $\mathrm{BSO} = \mathrm{BO}\langle 2\rangle$, $\mathrm{BSpin} = \mathrm{BO}\langle 4\rangle$, and $\mathrm{BString} = \mathrm{BO}\langle 8\rangle$. 

Furthermore, there exists a sequence $\mathrm{BO}\langle 8\rangle \to \mathrm{BO}\langle 4\rangle \to \mathrm{BO}\langle 2\rangle \to \mathrm{BO}$, where each map in this sequence is a fibration. From the unstable variation of the Postnikov tower in \cite{BuNa} (see page 44), we can obtain the obstructions for lifting a map $X \to \mathrm{BO}$ to $\mathrm{BO}\langle 8\rangle$:
\begin{lemma}\label{BO8stru}
Let $\eta$ be a stable vector bundle over $X$. If $w_2(\eta) = w_1(\eta) = 0$, $H^4(X)$ is torsion-free, and $p_1(\eta) = 0$, then the map $\eta$ admits a $\mathrm{BO}\langle 8\rangle$-structure.
\end{lemma}

Recall $H^1(M)=0$ and $H^4(M)=\mb{Z}$. By Lemma \ref{BO8stru}, the virtual bundle $\mathscr{N}_{{M}}-f^\ast \xi$ admits a $\mathrm{BO}\langle 8\rangle$-structure. In other words, the classifying map of the virtual bundle $\mathscr{N}_{{M}}-f^\ast \xi$ has a lift to $\mm{BO}\langle 8\rangle$, namely $\nu: M\to \mm{BO}\langle 8\rangle$. Let $\gamma_8$ be the universal bundle over $\mm{BO}\langle 8 \rangle$, i.e. the pullback bundle of the universal bundle over $\mm{BO}$ through the fibration $\mm{BO}\langle 8 \rangle\to \mm{BO}$. Then we have $\nu^\ast\gamma_8\cong \mathscr{N}_{{M}}-f^\ast \xi$.

Let ${B}=\mm{BO}\langle 8 \rangle \times  \mb{CP}^\infty$. There is a product bundle, namely $\gamma_8\times \xi$, over $B$.
Its classifying map, denoted by
$\pi(\xi): {B}\to \mm{BO}$, can be factored as the composition $B\hookrightarrow E_{\pi} \to \mm{BO}$ of a homotopy equivalence and a fibration by \cite[Proposition 4.64, pp.407]{Hatcher}. Hence we regard the map $\pi(\xi)$ as a fibration.

Now we claim that the manifold $M$ admits a normal $B$-structure. Consider
the composition 
$$(\nu,f):{M}\stackrel{\Delta}{\longrightarrow }{M}\times {M}\stackrel{\nu\times f}{\longrightarrow}\mm{BO}\langle 8 \rangle \times \mb{CP}^\infty$$ 
 where $\Delta$ is the diagonal map. Recalling the Whitney sum of vector bundles, we have $$\mathscr N_{{M}}\cong \nu^\ast\gamma_8\oplus f^\ast\xi\cong (\nu,f)^\ast(\gamma_8\times \xi)$$ 
 Hence, the map $(\nu,f):{M}\to B$ is a lift of the classifying map of the stable normal bundle of $M$. 
 Thus, the claim follows. Indeed, $(\nu,f)$ induces isomorphisms on $H_i$ for $0\le i\le 6$. 

All $m$-dimensional smooth manifolds that possess normal 
$B$-structures, when considered under the cobordant relation and with the disjoint union serving as the operation, form a bordism group
 $\Omega_m^{\mm{O}\langle 8 \rangle}(\xi)$ \cite[pp.226]{Sw}. This bordism group is isomorphic to the homotopy group $\pi_n(\mm{MO}\langle 8 \rangle\wedge \mm{M}\xi)$ by the Pontryagin-Thom isomorphism, where $\mm{M}\xi$ and $\mm{MO}\langle 8\rangle$ are the Thom spectra of the bundles $\xi$ and $\gamma_8$, respectively.
 A bordism class of $\Omega_m^{\mm{O}\langle 8 \rangle}(\xi)$ is denoted by a pair $[\mc{M},\bar{\mathscr N}_\mc{M}]$ where $\mc{M}$ is a $m$-dimensional manifold, $\bar{\mathscr N}_\mc{M}$ is a normal $B$-structure of $\mc{M}$. 
It should be noted that the group $\Omega_m^{\mm{O}\langle 8 \rangle}(\xi)$ depends on the bundle $\xi$. 

 Let ${M}$, ${M}'$ be simply connected, closed, smooth 13-dimensional manifolds with the same cohomology ring as $H^*(\mathbb{CP}^3 \times S^7)$. When $p_1({M})=p_1({M}')$, we can choose the same bundle $\gamma_8\times \xi$ over $B$ such that the bordism classes $[{M},(\nu,f)]$ and $[{M}',(\nu',f')]$ lie in the same bordism group $\Omega_{13}^{\mm{O}\langle 8 \rangle}(\xi)$.
In the following sections, we will consider whether the bordism classes are equal.

\section{The filtrations of $\Omega_{13}^{\mm{O}\langle 8 \rangle}(\xi)$}\label{FiltrationforM}

Let $M$ be a simply connected, closed, smooth 13-dimensional manifolds with the same cohomology ring as $H^*(\mathbb{CP}^3 \times S^7)$. From Section \ref{normal5}, we obtain  bordism groups
 $\Omega_\ast^{\mm{O}\langle 8 \rangle}(\xi)$ associated with the manifold $M$. The pair $[M,(\nu,f)]$ represents a bordism class in $\Omega_{13}^{\mm{O}\langle 8 \rangle}(\xi)$.
 

There is an Atiyah-Hirzebruch spectral sequence (AHSS) as follows
 $$E_2^{p,q}\cong H_p(\mm{M}\xi
 ;\pi_q(\mm{MO}\langle 8\rangle))\Longrightarrow  \pi_{p+q}(\mm{MO}\langle 8\rangle\wedge \mm{M}\xi)\cong \Omega_{p+q}^{\mm{O}\langle 8 \rangle}(\xi).$$
 This spectral sequence admits a sequence of filtrations
 $$\Omega_m^{\mm{O}\langle 8 \rangle}(\xi)=F^{m,0}\supset F^{m-1,1}\supset\cdots \supset F^{0,m}$$
 satisfying $F^{i,m-i}/F^{i-1,m+1-i}\cong E_\infty^{i,m-i}$.
 
 In this section, we first prove the following proposition: 
 \begin{proposition}\label{filtraofM}
		 The bordism class $[M,(\nu,f)]$ lies in $F^{4,9}$.
	\end{proposition}
 
 

	

Now, we construct new bordism groups associated with the $M$. 

Recall Proposition \ref{Xconstruct}. For $p=2$, and the $M$, there exists a space $X$ with a map $g:M\to X$ inducing isomorphisms on $H^i$ for $0\le i\le 6$. Moreover, $H^\ast(X)=\mb{Z}[x]/kx^4$ where $\mm{gcd}(k,2)=1$.
We can take a suitable map $h:X\to \mb{CP}^\infty$ such that $h\circ g\simeq f:M\to \mb{CP}^\infty$.

Let $\eta=h^\ast\xi$. Then we take the product bundle $\gamma_8\times \eta$ over $A=\mm{BO}\langle 8\rangle \times X$.  All $m$-dimensional manifolds that possess normal 
$A$-structures, when considered under the cobordant relation and with the disjoint union serving as the operation, form a bordism group
  $$\Omega_{m}^{\mm{O}\langle 8\rangle}(\eta)\cong \pi_{m}(\mm{MO}\langle 8\rangle \wedge \mm{M}\eta)$$
  where $\mm{M}\eta$ is the Thom spectra of $\eta$. 
  It is easy to check that the pair $[M,(\nu,g)]$ represents a bordism class in $\Omega_{13}^{\mm{O}\langle 8\rangle}(\eta)$.
  
  There is also an AHSS as follows
 $$\bar E_2^{p,q}\cong H_p(\mm{M}\eta
 ;\pi_q(\mm{MO}\langle 8\rangle))\Longrightarrow  \pi_{p+q}(\mm{MO}\langle 8\rangle\wedge \mm{M}\eta)\cong \Omega_{p+q}^{\mm{O}\langle 8 \rangle}(\eta).$$
 This spectral sequence also admits a sequence of filtrations
 $$\Omega_m^{\mm{O}\langle 8 \rangle}(\eta)=\bar F^{m,0}\supset \bar F^{m-1,1}\supset\cdots \supset \bar F^{0,m}$$
 satisfying $\bar F^{i,m-i}/\bar F^{i-1,m+1-i}\cong \bar E_\infty^{i,m-i}$.
 \begin{lemma}\label{newbord}
 	$\bar F^{13,0}\cong {_2}\bar F^{13,0}\oplus {_3}\bar F^{13,0}\oplus G$, and  ${_2}\bar F^{13,0}\subset \bar F^{4,9}$ where ${_2}\bar F^{13,0}$ is the $2$-primary part of $\bar F^{13,0}$, ${_3}\bar F^{13,0}$ is the $3$-primary part of $\bar F^{13,0}$, the order of any element in $G$ is coprime to $2$ and $3$.
 \end{lemma} 
\begin{proof}
	In dimension $i \le 14$ \cite{Giambalvo,HoRa1995}, $\pi_i(\mm{MO} \langle  8 \rangle )$ is as follows.
\begin{center}
\begin{tabular}{|c|c|c|c|c|c|c|c|c|c|c|c|c|c|c|c|}
\hline
$i$ & 0&1& 2& 3&4& 5& 6&7&8&9\\
\hline
$\pi_i(\mm{MO} \langle  8\rangle )$ & $\mathbb{Z}$ & $\mathbb{Z}_2$ & $\mathbb{Z}_2$ & $\mathbb{Z}_{24}$& 0& 0& $\mathbb{Z}_2$ & 0& $ \mathbb{Z}\oplus \mathbb{Z}_2$ & $\mathbb{Z}_2\oplus \mathbb{Z}_2 $\\
\hline

\end{tabular}
\end{center}
\begin{center}
\begin{tabular}{|c|c|c|c|c|c|c|c|c|c|c|c|c|c|c|c|}
\hline
$i$ &10&11&12&13&14\\
\hline
$\pi_i(\mm{MO} \langle  8 \rangle )$  & $\mathbb{Z}_6 $ & 0& $\mathbb{Z}$ & $\mathbb{Z}_3$ & $\mathbb{Z}_2$\\
\hline

\end{tabular}
\end{center}
By Thom isomorphism, $H_i(\mm{M}\eta)\cong H_i(X)$ is $\mb{Z}_k$ or $0$ for $i\ge 7$ where $\mm{gcd}(k,2)=1$. Combining these data, we have that $\bar F^{13,0}$ is finite, and $\bar F^{13,0}\cong {_2}\bar F^{13,0}\oplus {_3}\bar F^{13,0}\oplus G$ where the order of any element in $G$ is coprime to $2$ and $3$. 

 Since ${_2}\bar E_2^{i,13-i}=0$, thus ${_2}\bar E_\infty^{i,13-i}=0$ for $5\le i\le 13$. By $$\bar F^{i,13-i}/\bar F^{i-1,14-i}\cong \bar E_\infty^{i,13-i}$$ we have ${_2}\bar F^{13,0}={_2}\bar F^{4,9}$.
\end{proof}

\begin{proof}[Proof of Proposition \ref{filtraofM}-step 1]
Recall that the homotopy groups of $\mm{MO}\langle 8\rangle$ and $H_i(\mm{M}\xi)\cong H_i(\mb{CP}^\infty)$. By the AHSS, we have $$\Omega_{13}^{\mm{O}\langle 8\rangle}(\xi)\cong {_2}F^{13,0}\oplus {_3}F^{13,0}$$
Let $(\beta_1,\beta_2)\in {_2}F^{13,0}\oplus {_3}F^{13,0}$
represent the bordism class $[M,(\nu,f)]$.

By Lemma \ref{newbord}, let 
$$(\alpha_1,\alpha_2,\alpha_3)\in \bar F^{13,0}\cong {_2}\bar F^{13,0}\oplus {_3}\bar F^{13,0}\oplus G$$
represent the bordism class $[M,(\nu,g)]$ where ${_2}\bar F^{13,0}\subset \bar F^{4,9}$. 

 The map $\mm{id}\times h:A=\mm{BO}\langle 8\rangle \times X\to B=\mm{BO}\langle 8\rangle \times \mb{CP}^\infty$
	induces a homomorphism between bordism groups
	$$h_\ast:\Omega_{13}^{\mm{O}\langle 8\rangle}(\eta)\to \Omega_{13}^{\mm{O}\langle 8\rangle}(\xi)$$

	Since $f\simeq h\circ g$, $h_\ast([M,(\nu,g)])=[M,(\nu,f)]$. Thus $h_\ast(\alpha_3)=0$, $h_\ast(\alpha_1)=\beta_1$, $h_\ast(\alpha_2)=\beta_2$. 
 By the naturality of AHSS and $\alpha_1\in \bar F^{4,9}$, we have $\beta_1\in F^{4,9}$. 
\end{proof}

For $p=3$, and the $M$, there exists a space $X'$ with a map $g':M\to X'$ inducing isomorphisms on $H^i$ for $0\le i\le 6$. Moreover, $H^\ast(X')=\mb{Z}[x]/k'x^4$ where $\mm{gcd}(k',3)=1$.
We can take a suitable map $h':X'\to \mb{CP}^\infty$ such that $h'\circ g'\simeq f:M\to \mb{CP}^\infty$. By the same process, we have another bordism group $\Omega_{13}^{\mm{O}\langle 8\rangle}(\eta')$ containing the bordism class $[M,(\nu,g')]$. For this bordism group, we also have the AHSS and its associated filtrations 
$$\Omega_{13}^{\mm{O}\langle 8\rangle}(\eta')=\bar F^{'13,0}\supset \cdots \bar F^{'13-i,i}\supset \cdots$$

Observing the cohomology ring of $X'$, we can also prove
\begin{lemma}\label{newbord2}
 	$\bar F^{'13,0}\cong {_2}\bar F^{'13,0}\oplus {_3}\bar F^{'13,0}\oplus G'$, and  ${_3}\bar F^{'13,0}\subset \bar F^{'4,9}$ where the order of any element in $G'$ is coprime to $2$ and $3$.
 \end{lemma} 
 
 \begin{proof}[Proof of Proposition \ref{filtraofM}-step 2]
 Recall	that $(\beta_1,\beta_2)\in {_2}F^{13,0}\oplus {_3}F^{13,0}$
represents the bordism class $[M,(\nu,f)]\in \Omega_{13}^{\mm{O}\langle 8\rangle}(\xi)$. Applying Lemma \ref{newbord2} and the induced homomorphism $h'_\ast:\Omega_{13}^{\mm{O}\langle 8\rangle}(\eta')\to \Omega_{13}^{\mm{O}\langle 8\rangle}(\xi)$, we have $\beta_2\in F^{4,9}$. This completes the proof.
 \end{proof}

Next we compute the subgroup $F^{4,9}\subset \Omega_{13}^{\mm{O}\langle 8\rangle}(\xi)$.

Take the restriction bundle $\xi|_{4}$ over $\mb{CP}^4$ of $\xi$. Let $\mm{M}\xi|_4$ be the Thom spectra of $\xi|_4$. There is a natural morphism between AHSS
\[
\xymatrix@C=.7cm{
  \tilde{E}_2^{p,q}={H}_p(\mm{M}\xi|_4;\pi_q(\mm{MO}\langle 8\rangle))\ar@{=>}[rr]^-{\tilde d_r}\ar[d]&  & \pi_{p+q}(\mm{MO}\langle 8\rangle\wedge \mm{M}\xi|_4)\ar[d]^-{} \\
{E}_2^{p,q}={H}_p(\mm{M}\xi;\pi_q(\mm{MO}\langle 8\rangle))\ar@{=>}[rr]^-{d_r}& & \pi_{p+q} ( \mm{MO}\langle 8\rangle\wedge \mm{M}\xi)
}
\]
Let $\tilde F^{p,q}$ be a filtration associated with the spectral sequence $\{\tilde E_r^{p,q},\tilde d_r\}$. In particular, $\tilde F^{p,q}/\tilde F^{p-1,q+1}\cong \tilde E_\infty^{p,q}$.
\begin{lemma}\label{Ftilde}
	$\tilde F^{13,0}=\tilde F^{4,9}$.
\end{lemma}
\begin{proof}
	This lemma follows from $\tilde E_2^{i,13-i}=0$ for $5\le i\le 13$.
\end{proof}

\begin{lemma}\label{morphtildeto}
	The homomorphism $\tilde{E}_\infty^{i,13-i}\to {E}_\infty^{i,13-i}$ is surjective for $0\le i\le 4$.
\end{lemma}
\begin{proof}
	Note that $\tilde{E}_2^{i,q}\cong {E}_2^{i,q}$ for $0\le i\le 4$, $q\ge 0$. By comparing two spectral sequences, we finish the proof.  
\end{proof}

\begin{proposition}\label{F49smooth}
	 $F^{4,9}\subset \Omega_{13}^{\mm{O}\langle 8\rangle}(\xi)$ has the structure as follows:
		\begin{enumerate}
		\item If $w_2(\xi)\ne 0$, $F^{4,9}=\mb{Z}_3$ or $0$. When $F^{4,9}=\mb{Z}_3$, its generator can be represented by an exotic $13$-sphere $\Sigma^{13}$.
		\item If $w_2(\xi)\ne 0$ and $p_1(\xi)\not\equiv 0\pmod 3$, $F^{4,9}=0$.
	\end{enumerate}
\end{proposition}
\begin{proof}
	Consider the following morphism of short exact sequences
	\[
\xymatrix@C=.7cm{
  0\ar[r]^-{}& \tilde F^{i,13-i}\ar[r]\ar[d] & \tilde F^{i+1,12-i}\ar[d]^-{}\ar[r]&\tilde E_\infty^{i+1,12-i}\ar[r]\ar[d]&0 \\
0\ar[r]^-{}&  F^{i,13-i}\ar[r] &  F^{i+1,12-i}\ar[r]& E_\infty^{i+1,12-i}\ar[r]&0
}
\]
where $0\le i\le 3$.
Suppose the left homomorphism is surjective, then by Lemma \ref{morphtildeto}, the middle homomorphism is surjective.
Since $F^{0,13}\cong E_\infty^{0,13}$ and $\tilde F^{0,13}\cong \tilde E_\infty^{0,13}$, we can show that $\tilde F^{i,13-i}\to  F^{i,13-i}$ is surjective for $i=1,2,3,4$ one by one.  Therefore, by Lemma \ref{Ftilde}, the homomorphism $\pi_{13}(\mm{MO}\langle 8\rangle\wedge \mm{M}\xi|_4)\to F^{4,9}\subset \pi_{13}(\mm{MO}\langle 8\rangle\wedge \mm{M}\xi)$ is surjective.
By Lemma 4.3, 4.4, 4.5 in \cite{Shen}, we finish the proof.
\end{proof}

\section{Bordism in PL category}\label{PLcateSec}

Let $\mm{BPL}$ be the classifying space of stable piecewise linear (PL) bundles. There is a natural map $\alpha:\mm{BO}\to \mm{BPL}$. 
We also take the $7$-connected cover $\mm{BPL}\langle 8\rangle$ of $\mm{BPL}$. The map $\alpha$ induces a map $T:\mm{BO}\langle 8\rangle\to \mm{BPL}\langle 8\rangle$. There is a universal bundle $\Gamma_8$ over $\mm{BPL}\langle 8\rangle$. 

 Let $B^{pl}=\mm{BPL}\langle 8\rangle\times \mb{CP}^\infty$ associated with the product bundle $\Gamma_8\times |\xi|$ where $|\xi|$ denote the PL bundle forgetting the structure of $\xi$ as a vector bundle. 
For PL manifolds, we have the same definition as Definition \ref{DefnorB}. All $m$-dimensional PL manifolds that possess normal 
$B^{pl}$-structures, when considered under the cobordant relation and with the disjoint union serving as the operation, form a bordism group
 $\Omega_m^{\mm{PL}\langle 8 \rangle}(|\xi|)$. 
 
 The map $T\times \mm{id}:\mm{BO}\langle 8\rangle\times \mb{CP}^\infty\to \mm{BPL}\langle 8\rangle\times \mb{CP}^\infty$ induces a natural homomorphism $T_\ast:\Omega_m^{\mm{O}\langle 8 \rangle}(\xi)\to \Omega_m^{\mm{PL}\langle 8 \rangle}(|\xi|)$. Obviously, the Thom spectra of $|\xi|$ is $\mm{M}\xi$. Hence $\Omega_m^{\mm{PL}\langle 8 \rangle}(|\xi|)\cong \pi_m(\mm{MPL}\langle 8\rangle\wedge \mm{M}\xi)$ where $\mm{MPL}\langle 8\rangle$ is the Thom spectra of $\Gamma_8$.
 
 \begin{proposition}\label{PLborM}
 	Let $M$ be a simply connected, closed, smooth $13$-manifold with the same cohomology ring as $H^*(\mathbb{CP}^3 \times S^7)$. If $M$ admits a normal $B$-structure as $(\nu,f):M\to B$ (cf. Section \ref{normal5}), then in PL category, the bordism class $[M,(T\circ \nu,f)]$ equals $0$ in $\Omega_{13}^{\mm{PL}\langle 8 \rangle}(|\xi|)$.   
 \end{proposition}
 
 The homomorphism $T_\ast$ induces a morphism between AHSS
\[
\xymatrix@C=.7cm{
  {E}_2^{p,q}={H}_p(\mm{M}\xi;\pi_q(\mm{MO}\langle 8\rangle))\ar@{=>}[rr]^-{ d_r}\ar[d]&  & \pi_{p+q}(\mm{MO}\langle 8\rangle\wedge \mm{M}\xi)\ar[d]^-{} \\
\mb{E}_2^{p,q}={H}_p(\mm{M}\xi;\pi_q(\mm{MPL}\langle 8\rangle))\ar@{=>}[rr]^-{\mm{d}_r}& & \pi_{p+q} ( \mm{MPL}\langle 8\rangle\wedge \mm{M}\xi)
}
\]
The bottom spectral sequence admits a sequence of filtrations
$$\Omega_m^{\mm{PL}\langle 8\rangle}(|\xi|) =\mb{F}^{m,0}\supset \mb{F}^{m-1,1}\supset \cdots \supset \mb{F}^{m-i,i}\supset \cdots $$
where $\mb{F}^{n-i,i}/\mb{F}^{n-i-1,i+1}\cong \mb{E}_\infty^{n-i,i}$.
\begin{lemma}\label{PLF49}
	$\mb{F}^{4,9}=0$.
\end{lemma}
\begin{proof}
Recall $H_i(\mm{M}\xi)=H_i(\mb{CP}^\infty)$.
	By \cite{Shen2}, $\pi_i(\mm{MPL}\langle 8\rangle)=0$ for $i=9,11,13$. Hence, $\mb{E}_2^{13-i,i}=0$, and thus $\mb{E}_\infty^{13-i,i}=0$ for $9\le i\le 13$. By $\mb{F}^{13-i,i}/\mb{F}^{12-i,i+1}\cong \mb{E}_\infty^{13-i,i}$, we finish the proof.
\end{proof}

\begin{proof}[Proof of Proposition \ref{PLborM}]
By Proposition \ref{filtraofM}, $$[M,(\nu,f)]\in F^{4,9}\subset \Omega_{13}^{\mm{O}\langle 8\rangle}(\xi)$$
By the naturality of AHSS, $[M,(T\circ\nu,f)]\in \mb{F}^{4,9}\subset \Omega_{13}^{\mm{PL}\langle 8\rangle}(|\xi|)$. By Lemma \ref{PLF49}, we complete this proof.	
\end{proof}

\section{Proofs of Theorems \ref{1.1} and \ref{1.3}}\label{proofmain}

Following the surgery in \cite{Shen} and its analogous version in PL category, we have
\begin{lemma}\label{hcobordism}
	If $[M,(\nu,f)]=[M',(\nu',f')]\in \Omega_{13}^{\mm{O}\langle 8\rangle}(\xi)$, $M$ is diffeomorphic to $M'$. If $[M,(T\circ\nu,f)]=[M',(T\circ\nu',f')]\in \Omega_{13}^{\mm{PL}\langle 8\rangle}(|\xi|)$, $M$ is PL homeomorphic to $M'$
\end{lemma}

Proofs of (1), (2) and (3) of Theorem \ref{1.1}:
	 Assume $p_1(M)=p_1({M}')$, then the bordism classes $[M,(\nu,f)]$ and $[M',(\nu',f')]$ lie in $\Omega_{13}^{\mm{O}\langle 8\rangle}(\xi)$. Moreover, $[M,(T\circ\nu,f)]$, $[M',(T\circ\nu',f')]\in \Omega_{13}^{\mm{PL}\langle 8\rangle}(|\xi|)$.
	 
By Proposition \ref{PLborM}, we have
$$[M,(T\circ\nu,f)]=[M',(T\circ\nu',f')]=0\in \Omega_{13}^{\mm{PL}\langle 8\rangle}(|\xi|)$$
Lemma \ref{hcobordism} implies that $M$ is PL homeomorphic to $M'$.

If $w_2(M)=w_2(M')\ne 0$, then $w_2(\xi)\ne 0$. 
By Proposition \ref{filtraofM} and \ref{F49smooth} (1), there exists a homotopy $13$-sphere $\Sigma^{13}$ such that $M$ is cobordant to $M'\#\Sigma^{13}$ in $\Omega_{13}^{\mm{O}\langle 8\rangle}(\xi)$. By Lemma \ref{hcobordism}, $M$ is  diffeomorphic to $M'\#\Sigma^{13}$. 	
 	Especially, if $p_1({M})=p_1(M')\not\equiv 0\mod 3$, then $p_1(\xi)\not\equiv 0\mod 3$. By Proposition \ref{F49smooth} (2), $M$ is cobordant to $M'$ in $\Omega_{13}^{\mm{O}\langle 8\rangle}(\xi)$. By Lemma \ref{hcobordism}, $M$ is  diffeomorphic to $M'$.
	 $\Box$

Proof of (4) of Theorem \ref{1.1}:
By \cite{Shen}, every smooth manifold $M$ is homeomorphic to an $S^7$-bundle over $\mb{CP}^3$. In particular, the structure group of this bundle is $\mm{O}(8)$. Applying the homotopy classification of \cite{Shen}, we finish the proof.  
 $\Box$

Proof of (5) of Theorem \ref{1.1}:
Let $M$, $N$ be simply connected, closed, smooth manifolds with the same cohomology ring as $\mb{CP}^3\times S^7$. Moreover, let $p_1(M)=p_1(N)\equiv 4\pmod {24}$.

By Theorem \ref{1.1} (1) and (4), there exist a homeomorphism $h:M\to N$, homotopy equivalences $\alpha:M\to \mb{CP}^3\times S^7$ and $\beta:N\to \mb{CP}^3\times S^7$ such that $\beta \circ h\simeq \alpha$. Consider the surgery exact sequence
$$\mm{P}_{14}\stackrel{\omega}{\to }\bold S(\mb{CP}^3\times S^7)\stackrel{\bold{q}}{\to}[\mb{CP}^3\times S^7,G/O]\to 0$$ 
The pairs $(M,\alpha)$ and $(N,\beta)$ denotes two classes in $S(\mb{CP}^3\times S^7)$.
Following \cite{BasuKaSa}, $\bold q(\alpha)-\bold q(\beta)$ lies in the image of $[\mb{CP}^3\times S^7,\mm{TOP/O}]$ under the homomorphism $[\mb{CP}^3\times S^7, \mm{TOP/O}]\to [\mb{CP}^3\times S^7,G/O]$ where the minus $-$ corresponds the subtraction of the group $[\mb{CP}^3\times S^7,G/O]$. By \cite[Lemma 3.3 (vii)]{BasuKaSa}, this image is $\mb{Z}_2$ and generated by $[\nu^3]$.
Therefore, $$\bold q(\alpha)=\bold q(\beta)+s[\nu^3], \quad \text{$s=0$ or $1$}$$

If $s=0$, then by the surgery exact sequence, $M$ is diffeomorphic to $N\#\Sigma^{13}$ where $\Sigma^{13}\in \mm{bP}_{14}=0$. Hence, $M$ is diffeomorphic to $N$.

If $s=1$, then we consider the pair $(M,f_{\nu^2}\circ \alpha)$ where $f_{\nu^2}:\mb{CP}^3\times S^7\to \mb{CP}^3\times S^7$ is a self-homotopy equivalence (see \cite[Lemma 3.31]{BasuKaSa}). Thus 
$$\bold q(f_{\nu^2}\circ \alpha)=\bold q(f_{\nu^2})+f_{\nu^2}^{-1\ast}\bold q(\alpha)=[\nu^3]+f_{\nu^2}^{-1\ast}\bold q(\alpha)$$
We claim that $f_{\nu^2}^{-1\ast}\bold q(\alpha)= \bold q(\alpha)$. 
Then $$\bold q(f_{\nu^2}\circ \alpha)=[\nu^3]+\bold q(\alpha)=\bold q(\beta)$$
Hence, by the surgery exact sequence, $M$ is diffeomorphic to $N$. 

Now we prove the claim.

Step (1): We compute the isomorphism 
$$f_{\nu^2}^{-1\ast}:[\mb{CP}^3\times S^7,G/\mm{TOP}]\to [\mb{CP}^3\times S^7,G/\mm{TOP}]$$
There is a spectral sequence
$$E_2^{p,q}=H^p(\mb{CP}^3\times S^7;\pi_{-q}(G/\mm{TOP}))\Longrightarrow K^{p+q}_{G/\mm{TOP}}(\mb{CP}^3\times S^7)$$ 
where $K^{0}_{\mm{TOP/O}}(\mb{CP}^3\times S^7)=[\mb{CP}^3\times S^7,G/\mm{TOP}]$. 

Let $H^\ast(\mb{CP}^3\times S^7)=\mb{Z}[x,y]/(x^4,y^2)$ where $\mm{deg}(x)=2$, $\mm{deg}(y)=7$. Since $f_{\nu^2}^{-1}$ is a homotopy equivalence, $f_{\nu^2}^{-1\ast}(x^i)=(\pm x)^i$ for $1\le i\le 3$. Moreover, $f_{\nu^2}^{-1\ast}$ induces isomorphisms on $E_\infty^{p,q}$ for any $p,q$.
Recall that $\pi_{2i+1}(G/\mm{TOP})=0$, $\pi_{4i+2}(G/\mm{TOP})=\mb{Z}_2$, $\pi_{4i}(G/\mm{TOP})=\mb{Z}$. We have: 
\begin{enumerate}
	\item $f_{\nu^2}^{-1\ast}:E_2^{p,-p}\to E_2^{p,-p}$ is the identity for $p=0,2,4,6$.
	\item $E_2^{p,-p}=0$ for $p\ne 0,2,4,6$.
\end{enumerate} 
Therefore, $f_{\nu^2}^{-1\ast}:E_\infty^{p,-p}\to E_\infty^{p,-p}$ is the identity for $p=0,2,4,6$. Since every element in $[\mb{CP}^3\times S^7,G/\mm{TOP}]$ can be represented by the extension of  certain elements in $E_\infty^{p,-p}$ for $p\ge 0$, we have that $f_{\nu^2}^{-1\ast}:[\mb{CP}^3\times S^7,G/\mm{TOP}]\to [\mb{CP}^3\times S^7,G/\mm{TOP}]$ is the identity.

Step (2): By the following commutative diagram 
 \[
\xymatrix@C=.7cm{
  [\mb{CP}^3\times S^7,G/\mm{O}]\ar[rr]^-{f_{\nu^2}^{-1\ast}}\ar[d]^-{T_\ast}&  & [\mb{CP}^3\times S^7,G/\mm{O}]\ar[d]^-{T_\ast} \\
[\mb{CP}^3\times S^7,G/\mm{TOP}]\ar[rr]^-{f_{\nu^2}^{-1\ast}=\mm{id}}& & [\mb{CP}^3\times S^7,G/\mm{TOP}]
}
\] 
we have $T_\ast\circ f_{\nu^2}^{-1\ast}\bold q(\alpha)=T_\ast(\bold q(\alpha))$. By the exact sequence for the fibration $\mm{TOP/O}\to G/\mm{O}\to G/\mm{TOP}$, $f_{\nu^2}^{-1\ast}\bold q(\alpha)=\bold q(\alpha)+k[\nu^3]$ where $k=0$ or $1$. 
Consider the spectral sequence
$$\mm{E}_2^{p,q}=H^p(\mb{CP}^3\times S^7;\pi_{-q}(G/\mm{O}))\Longrightarrow K^{p+q}_{G/\mm{O}}(\mb{CP}^3\times S^7)$$ 
where $K^{0}_{G/\mm{O}}(\mb{CP}^3\times S^7)=[\mb{CP}^3\times S^7,G/\mm{O}]$.
By \cite[Lemma 3.31]{BasuKaSa}, $[\nu^3]$ is provided by $\mm{E}_2^{9,-9}$ in the spectral sequence. Thus $f_{\nu^2}^{-1\ast}\bold q(\alpha)$ and $\bold q(\alpha)$ can be represented by the extension of the same elements in $\mm{E}_\infty^{p,-p}$ for $p\ge 0$ except $p=9$.

Since $\pi_9(G/\mm{O})=\mb{Z}_2\oplus \mb{Z}_2$, $f_{\nu^2}^{-1\ast}:\mm{E}_r^{9,-9}\to \mm{E}_r^{9,-9}$ is the identity for $r\ge 2$. Therefore, $k=0$, and thus $f_{\nu^2}^{-1\ast}\bold q(\alpha)=\bold q(\alpha)$. 
 $\Box$

 \begin{proof}[Proof of Theorem \ref{1.3}]
 	By \cite{Shen}, Theorem \ref{1.3} follows.
 \end{proof}
 
\section{The $S^1_{\bold{a,b}}$ quotient of $S^7\times S^7$}\label{constructforXab}

\begin{proof}[Proof of Proposition \ref{freeS1action}]
	The $S^1_{\bold{a,b}}$-action is smooth. We only need to check that this action does not admit fixed points.
	
	Assume that there exist $\theta\in S^1$ and $(x,y)\in S^7\times S^7$ such that $\theta\cdot (x,y)=(x,y)$. This means $e^{\bold{i}a_i\theta}x_i=x_i$ and $e^{\bold{i}b_i\theta}y_i=y_i$ for each $i$. Since $\Sigma_i |x_i|^2=\Sigma_i |y_i|^2=1$, there exist $x_{i_1}\ne 0$ and $y_{i_2}\ne 0$ where $1\le i_1,i_2\le 4$. Then we have $e^{\bold{i}a_{i_1}\theta}=e^{\bold{i}b_{i_2}\theta}=1$. This implies $\theta=2n\pi/a_{i_1}$ for some $n\in \mb{Z}$, and thus $2n\pi b_{i_2}/a_{i_1}=2m\pi$ for some $m\in \mb{Z}$. Hence $ma_{i_1}=nb_{i_2}$. By $\mm{gcd}(a_{i_1},b_{i_2})=1$, $n/a_{i_1}$ is an integer. So $\theta$ is the unit of $ S^1$. This finishes the proof.  
\end{proof}

Assume that an $S^1_{\bold{a,b}}$-action is free, the quotient space $M_{\bold{a, b}}$ induced by the $S^1_{\bold{a,b}}$-action on $S^7\times S^7$ is a smooth 13-manifold, and admits a canonical principle $S^1$-bundle:
\begin{equation}
	S^1\to S^7\times S^7\to M_{\bold{a,b}} \label{canobundMab}
\end{equation}    
Applying the Gysin sequence \cite[pp.438]{Hatcher} on the above bundle, we have
\begin{lemma}
	 $H^\ast(M_{\bold{a,b}})\cong H^\ast(\mb{CP}^3\times S^7)$ as rings.
\end{lemma}

 There exits a fibre bundle:
\begin{equation}
	S^7\times S^7\hookrightarrow S^\infty\times_{S^1}(S^7\times S^7)\to \mb{CP}^\infty \label{fibrebun-S7S7overCPin}
\end{equation}
where the quotient space $S^\infty\times_{S^1}(S^7\times S^7)$ is induced by the $S^1$-action on $S^\infty\times (S^7\times S^7)$: the right $S^1$-action on $S^\infty$ is the Hopf action, the left $S^1$-action on $S^7\times S^7$ is the $S^1_{\bold {a,b}}$-action. A straightforward verification shows that $S^\infty\times_{S^1}(S^7\times S^7)$ is homotopy equivalent to $M_{\bold{a,b}}$.  

Now we consider the restriction bundle of \eqref{fibrebun-S7S7overCPin} over $\mb{CP}^3$:
\begin{equation}
	\eta_{\bold{a,b}}:S^7\times S^7\hookrightarrow S^7\times_{S^1}(S^7\times S^7)\to \mb{CP}^3 \label{etaab-bundCP4}
\end{equation}
where the right $S^1$-action on $S^7$ is the Hopf action, the left $S^1$-action on $S^7\times S^7$ is the $S^1_{\bold {a,b}}$-action. The space $S^7\times_{S^1}(S^7\times S^7)$ is a $20$-dimensional manifold, denoted by $X_{\bold{a,b}}$. Applying the Serre spectral sequence on the bundle $\eta_{\bold{a,b}}$, we have
\begin{lemma}
	The bundle projection $X_{\bold{a,b}}\to \mb{CP}^3$ induces isomorphisms on $H^k$ for $k\le 6$. 
\end{lemma}

Consider the following bundle map
\[
\xymatrix@C=0.8cm{
S^1\ar[d]^-{\Delta}\ar[r] &S^7\times (S^7\times S^7)\ar[r]\ar[d]^-{\mm{id}}&X_{\bold{a,b}} \ar[d]^-{\mathfrak {p}}\\
S^1\times S^1 \ar[r]^-{}& S^7\times (S^7\times S^7)\ar[r] &  \mb{CP}^3\times M_{\bold{a,b}}
}
\]
where $\Delta(\theta)=(\theta,\theta)\in S^1\times S^1$. 
\begin{lemma}\label{tangentXtoCPandM}
	(1) The tangent bundle $\tau(X_{\bold{a,b}})$ of $X_{\bold{a,b}}$ is isomorphic to the Whitney sum $\epsilon^1 \oplus \mathfrak{ p}^\ast(\tau(\mb{CP}^3\times M_{\bold{a,b}}))$ where $\epsilon^1$ is the trivial $\mb{R}^1$-bundle over $X_{\bold{a,b}}$, $\tau(\mb{CP}^3\times M_{\bold{a,b}})$ is the tangent bundle of $\mb{CP}^3\times M_{\bold{a,b}}$.
	
	(2) Let $j_1:\mb{CP}^3\times M_{\bold{a,b}}\to \mb{CP}^3$ and $j_2:\mb{CP}^3\times M_{\bold{a,b}}\to M_{\bold{a,b}}$ be the projections. The composition of the map $\mathfrak p$ and the projection $j_i$ ($i=1,2$) induces isomorphisms on $H^k$ for $k\le 6$.
\end{lemma}
\begin{proof}
	From the above bundle map, we have a principle $S^1$-bundle
$$S^1\to X_{\bold{a,b}}\stackrel{\mathfrak{ p}}{\to } \mb{CP}^3\times M_{\bold{a,b}}$$
This bundle implies the item (1). A straightforward computation shows the item (2).
\end{proof} 

Recall the cohomology rings of $\mb{CP}^3$, $M_{\bold{a,b}}$ and $X_{\bold{a,b}}$. Let $\mathcal V=\mb{CP}^3$, $M_{\bold{a,b}}$ or $X_{\bold{a,b}}$. Note that $x^2=-x\cup -x$ where $x$ is a generator of $H^2(\mathcal V)=\mb{Z}$.
We consistently choose the cup product $x^2$ as the generator of $H^4(\mathcal V)=\mb{Z}$. Consequently, we can represent the first Pontrjagin class  $p_1(\mathcal V)$ of $\mathcal V$ using just an integer. Since the n-th Stiefel-Whitney class $w_n\in H^n(-;\mb{Z}_2)=\mb{Z}_2$ for $n=2,4$, we can denote it by $0$ or $1\in \mb{Z}_2$.

\begin{lemma}\label{firstPonclass-X}
(1) $p_1(X_{\bold{a,b}})=4+\Sigma_i(a_i^2+b_i^2)\in H^4(X_{\bold{a,b}})$
\item (2) $w_2(X_{\bold{a,b}})\equiv \Sigma_i(a_i^2+b_i^2) \pmod 2$
\item (3) $w_4(X_{\bold{a,b}})\equiv \Sigma_i(a_i^2+b_i^2) \pmod 2$	
\end{lemma}
\begin{proof}
	Recall the $S^1_{\bold{a,b}}$-action on $S^7\times S^7$. Then we have the following bundle maps:
	\[
\xymatrix@C=0.8cm{
S^7\times S^7\ar[d]^-{}\ar[r] &\mb{C}^4\times S^7\ar[r]\ar[d]^-{}&\mb{C}^4\times \mb{C}^4 \ar[d]^-{}\\
S^7\times_{S^1}(S^7\times S^7)\ar[d] \ar[r]^-{i_1}& S^7\times_{S^1} (\mb{C}^4\times S^7)\ar[r]^-{i_2}\ar[d] &  S^7\times_{S^1} (\mb{C}^4\times \mb{C}^4)\ar[d]^-{\pi}\\
\mb{CP}^3\ar@{=}[r]& \mb{CP}^3\ar@{=}[r] &\mb{CP}^3
}
\]
Let $Y_1=S^7\times_{S^1} (\mb{C}^4\times S^7)$, $Y_2=S^7\times_{S^1} (\mb{C}^4\times \mb{C}^4)$. We have 
$$i_1^{\ast}(\tau(Y_1))\cong \tau(X_{\bold{a,b}})\oplus \epsilon^1(X_{\bold{a,b}})\quad i_2^{\ast}(\tau(Y_2))\cong \tau(Y_1)\oplus \epsilon^1(Y_1)$$
Hence $p_1(X_{\bold{a,b}})=(i_2\circ i_1)^\ast p_1(Y_2)$.
Since $i_1$ and $i_2$ induce isomorphisms on $H^k$ for $k\le 6$, we only need to compute $p_1(Y_2)$. 

Let $\gamma$ denote the right-hand complex $8$-dimension vector bundle in the above diagram. Let $|\gamma|$ be the underlying real bundle of $\gamma$. We have $\tau(Y_2)\cong \pi^\ast(\tau(\mb{CP}^3)\oplus |\gamma|)$. 
Recall $p_1(\mb{CP}^3)=4\in H^4(\mb{CP}^3)$. We claim that $p_1(|\gamma |)=\Sigma_i(a_i^2+b_i^2)\in H^4(\mb{CP}^3)$. Therefore, $p_1(Y_2)=4+\Sigma_i(a_i^2+b_i^2)\in H^4(Y_2)$. This finishes the proof.

Finally, we prove the claim. 
Let $\gamma^\infty$ be the following bundle
$$\mb{C}^4\times \mb{C}^4\hookrightarrow S^\infty\times_{S^1} (\mb{C}^4\times \mb{C}^4)\to \mb{CP}^\infty$$
such that its restriction bundle over $\mb{CP}^4$ is $\gamma$. Note that $\gamma^\infty$ is the associated bundle of a principal $\mm{U}(8)$-bundle, denoted as $\Theta^{\infty}$. The principal bundle $\Theta^{\infty}$ has the form $$\mm{U}(8)\to S^\infty\times_{S^1} \mm{U}(8)\to \mb{CP}^\infty$$
where the right $S^1$-action on $S^\infty$ is the Hopf action, the left $S^1$-action on $\mm{U}(8)$ is defined by the embedding 
$$S^1\hookrightarrow \mm{U}(8):\theta\to \mm{diag}(e^{\mm{i}a_1\theta},\cdots,e^{\mm{i}a_4\theta},e^{\mm{i}b_1\theta},\cdots,e^{\mm{i}b_4\theta})$$
and the multiplication in $\mm{U}(8)$. 

Let $\kappa:\mb{CP}^\infty \to \mm{BU}(8)$ be the classifying map of $\Theta^\infty$. By the following homotopy fibration,
$$S^1\hookrightarrow \mm{U}(8)\to \mm{U}(8)/S^1\simeq S^\infty\times_{S^1} \mm{U}(8)  \to \mm{B}S^1=\mb{CP}^\infty \stackrel{\kappa}{\longrightarrow }\mm{BU}(8)$$
${\kappa}$ is induced by the above embedding $S^1\hookrightarrow  \mm{U}(8).$ 

Now we recall the classical characteristic theory for convenience.
Let $G$ be a Lie group, and $T^n \subset G$ be a maximal torus in $G$ with induced map $\mathbf i: \mm{BT}^n\to \mm{BG}$. Let $I_G$ be the ring of polynomials in $H^\ast(\mm{BT}^n)$ invariant under the Weyl group $W(G)$.
\begin{theorem}
[Borel's Theorem \cite{Borel1953}] If $H^\ast(G)$ and $H^\ast(G/T^n)$ are torsion-free, then the homomorphism $\mathbf i^\ast: H^\ast(\mm{BG}) \to H^\ast(\mm{BT}^n)$ is a monomorphism with range $I_G$.
\end{theorem}

As was shown in \cite{Borel1953}, the conditions of Borel's Theorem are satisfied for all classical groups.
Recall that $H^\ast(\mm{BU}(8))\cong \Bbb Z[c_1, \cdots, c_8]$ is a polynomial ring on the Chern classses $c_1, \cdots, c_8$. Applying Borel's Theorem to $\mm{U}(8)$, we have that
the homomorphism
$H^\ast(\mm{BU}(8))\to H^\ast(\mm{BT}^8)$
sends  $c_i$ to the elementary symmetric polynomial $\sigma_i(x_1,\cdots,x_8)$ of degree $i$ where $x_j\in H^2(\mm{BT}^8)$, $1\le j\le 8$, are  the generators.
The embedding from $S^1$ into the maximal torus $\mm{T}^8$ is parameterized by $(\bold{a,b})=(a_1,\cdots,a_4,b_1,\cdots,b_4)$. Therefore, $\kappa^*(c_i)= \sigma_i(\bold{a,b})x^i$ where $x\in H^2(\mm{B}{S}^1)\cong \Bbb Z$ is a  generator. 

Since the Chern classes of $\gamma^\infty$ satisfy
\begin{equation}
c_i(\gamma^\infty)=\kappa^*(c_i)=\sigma_i(\bold{a,b})x^i \label{Chernclass}	
\end{equation}
the first and second Chern classes of the complex vector bundle $\gamma$ are $c_1(\gamma)=\sigma_1(\bold{a,b})x$, $c_2(\gamma)=\sigma_2(\bold{a,b})x^2$ where $x$ is a generator of $H^2(\mb{CP}^3)$. Thus $p_1(|\gamma |)=(\sigma_1(\bold{a,b}))^2-2\sigma_2(\bold{a,b})=\Sigma_{i=1}^4(a_i^2+b_i^2)$.
\end{proof}

By Lemma \ref{tangentXtoCPandM} and \ref{firstPonclass-X}, we have
\begin{lemma}\label{firstPonclass-M}
	$p_1(M_{\bold{a,b}})=\Sigma_{i=1}^4(a_i^2+b_i^2)\in H^4(M_{\bold{a,b}})$.
\end{lemma}

\begin{proof}[Proof of Theorem \ref{classifyS1action}]
Let 
$$-\bold a=(-a_1,-a_2,-a_3,-a_4)\quad -\bold b=(-b_1,-b_2,-b_3,-b_4)$$
There is a bundle isomorphism
\[
\xymatrix@C=1.2cm{
S^1\ar[d]^-{}\ar[r]^-{-1} &S^1 \ar[d]^-{}\\
S^7\times S^7\ar[d] \ar[r]^-{\mm{id}}& S^7\times S^7\ar[d] \\
M_{\bold{a,b}}\ar[r]^-{\cong}&M_{\bold{-a,-b}}}
\]
Let $c:M_{\bold{a,b}}\to \mb{CP}^\infty$, $-c:M_{\bold{-a,-b}}\to \mb{CP}^\infty$ denote the classifying map of the left bundle (denoted by $\phi$) and the right bundle (denoted by $-\phi$) respectively.
Note that $c$ and $-c$ represent generators of $H^2(M_{\bold{a,b}})=\mb{Z}$ and $H^2(M_{\bold{-a,-b}})=\mb{Z}$ respectively.

 Then we have the following homotopy commutative diagram
\[
\xymatrix@C=1.2cm{
M_{\bold{a,b}}\ar[r]^-{\cong}\ar[d]^-{c}&M_{\bold{-a,-b}}\ar[d]^-{-c}\\
\mb{CP}^\infty \ar[r]^-{-1}& \mb{CP}^\infty
}
\]
where $-1:\mb{CP}^\infty\to \mb{CP}^\infty$ induces the cohomology homomorphism $(-1)^\ast=-\mm{id}$ on $H^2$.

	By Lemma \ref{firstPonclass-M} and Theorem \ref{1.1} (1), there is a homeomorphism $f:M_{\bold{a,b}}\to M_{\bold{\bar a,\bar b}}$.
	
	Let $\bar c:M_{\bold{\bar a,\bar b}}\to \mb{CP}^\infty$ be the classifying map of the principle $S^1$-bundle $\bar \phi$ corresponding the free $S^1_{\bold{\bar a,\bar b}}$-action on $S^7\times S^7$. Moreover, $\bar c$ also represents a generator of $H^2(M_{\bold{\bar a,\bar b}})$. Therefore, either $\bar c \circ f\simeq c$ or $\bar c \circ f\simeq -c$ holds.  
	
	If $\bar c \circ f\simeq c$, then there is a bundle isomorphism between the principle $S^1$-bundles $\phi$ and $\bar \phi$. This implies that there is a self-homeomorphism of $S^7\times S^7$ such that the $S^1_{\bold{a,b}}$, $S^1_{\bold{\bar a,\bar b}}$-actions are equivalent.
	
If $\bar c \circ f\simeq -c$, then we have
\[
\xymatrix@C=1.2cm{
M_{\bold{\bar a,\bar b}}\ar[r]^-{f^{-1}}\ar[d]^-{\bar c}&M_{\bold{a,b}}\ar[r]^-{\cong}\ar[d]^-{c}&M_{\bold{-a,-b}}\ar[d]^-{-c}\\
\mb{CP}^\infty\ar[r]^{-1}&\mb{CP}^\infty \ar[r]^-{-1}& \mb{CP}^\infty 
}
\]
there is a bundle isomorphism between the principle $S^1$-bundles $-\phi$ and $\bar \phi$. This implies that there is a self-homeomorphism of $S^7\times S^7$ such that the $S^1_{\bold{-a,-b}}$, $S^1_{\bold{\bar a,\bar b}}$-actions are equivalent.	

The smooth case of Theorem \ref{classifyS1action} follows by Theorem \ref{1.1} (3) and the same process as above.   
\end{proof}

\end{document}